\documentclass[12pt,amscd]{amsart}
\usepackage[all]{xy}
\usepackage{graphicx}
\usepackage{amsmath,amsxtra,amssymb,latexsym, amscd,amsthm}
\usepackage{indentfirst}
\usepackage[mathscr]{eucal}

\newtheorem{thm}{Theorem}[section]
\newtheorem{cor}[thm]{Corollary}
\newtheorem{lem}[thm]{Lemma}
\newtheorem{prop}[thm]{Proposition}
\theoremstyle{definition}
\newtheorem{defn}[thm]{Definition}

\newtheorem{rem}[thm]{Remark}

\DeclareMathOperator{\N}{\mathbb {N}}
\DeclareMathOperator{\Z}{\mathbb {Z}}

\DeclareMathOperator{\depth}{depth}
\DeclareMathOperator{\ass}{Ass}

\DeclareMathOperator{\diam}{diam}
\DeclareMathOperator{\dstab}{dstab}
\DeclareMathOperator{\Min}{Min}

\def\alb {\mathbf {\alpha}}
\def\btb {\mathbf {\beta}}
\def\ttb {\mathbf {\theta}}
\def\gmb {\mathbf {\gamma}}
\def\zv {\mathbf 0}

\def\x {\mathbf x}
\def\y {\mathbf y}
\def\mi {\mathfrak m}
\def\al {\alpha}
\def\bt {\beta}
\def\h {\widetilde{H}}

\begin{document}

\title[Stability of Depths of Powers of Edge Ideals] {Stability of Depths of Powers of Edge Ideals}

\author{Tran Nam Trung}
\address{Institute of Mathematics, VAST, 18 Hoang Quoc Viet, Hanoi, Viet Nam}
\email{tntrung@math.ac.vn}
\date{}

\dedicatory{}
\commby{}
\begin{abstract} Let $G$ be a graph and let $I := I (G)$ be its edge ideal. In this paper, we provide an upper bound of $n$ from  which $\depth R/ I(G)^n$ is stationary, and compute this limit explicitly. This bound is always achieved if $G$ has no cycles of length $4$ and every its connected component is either a tree or a unicyclic graph.
\end{abstract}

\maketitle
\section*{Introduction}

Let $R = K[x_1,\ldots, x_r]$ be a polynomial ring over a field $K$ and $I$ a homogeneous ideal in $R$. Brodmann \cite{B} showed that $\depth R/I^n$ is a constant for sufficiently large $n$. Moreover
$$\lim_{n\rightarrow \infty} \depth R/I^n \leqslant \dim R -\ell(I),$$
where $\ell(I)$ is the analytic spread of $I$. It was shown in \cite[Proposition $3.3$]{EH} that this is an equality when the associated graded ring of $I$ is Cohen-Macaulay.  We call the smallest number $n_0$ such that $\depth R/I^n = \depth R/I^{n_0}$ for all $n \geqslant n_0$, the {\it index of depth stability} of $I$, and denote this number by $\dstab(I)$. It is of natural interest to find a bound for $\dstab(I)$. As until now we only know effective bounds of $\dstab(I)$ for few special classes of ideals $I$, such as complete intersection ideals (see \cite{CN}), square-free Veronese ideals (see \cite{HH1}), polymatroidal ideals (see \cite{HQ}). In this paper we will study this problem for {\it edge ideals}.

From now on, every graph $G$ is assumed to be simple (i.e., a finite, undirected, loopless and without multiple edges) without isolated vertices on the vertex set $V(G)=[r]:=\{1,\ldots,r\}$ and the edge set $E(G)$ unless otherwise indicated.  We associate to $G$ the quadratic squarefree monomial ideal
$$I(G) = (x_ix_j \ |\ \{i,j\} \in E(G)) \subseteq R = K[x_1,\ldots, x_r]$$
which is called the edge ideal of $G$.

If $I$ is a polymatroidal ideal in $R$, Herzog and Qureshi proved that $\dstab(I) < \dim R$ and they asked whether $\dstab(I) < \dim R$ for all Stanley-Reisner ideals $I$ in $R$  (see \cite{HQ}). For a graph $G$, if every its connected component is nonbipartite, then we can see that $\dstab(I(G)) < \dim R$ from \cite {CMS}. In general,  there is not an absolute bound of $\dstab(I(G))$ even in the case $G$ is a tree (see \cite{MO}). In this paper we will establish a bound of $\dstab(I(G))$ for any graph $G$. In particular, $\dstab(I(G))  < \dim R$.

The first main result of the paper shows that the limit of the sequence $\depth R/I(G)^n$ is the number $s$ of connected bipartite components of $G$ and $\depth R/I(G)^n$ immediately becomes constant once it reaches the value $s$. Moreover, $\dstab(I(G))$ can be obtained via its connected components.

\medskip

\noindent {\bf Theorem $\ref{MP}$}. {\it Let $G$ be a graph with $p$ connected components $G_1,\ldots, G_p$. Let $s$ be the number of connected bipartite components of $G$. Then
\begin{enumerate}
\item $\min\{\depth R/I(G)^n\mid n\geqslant 1\} = s$.
\smallskip
\item $\dstab(I(G)) = \min\{n\geqslant 1 \mid \depth R/I(G)^n = s\}$.
\smallskip
\item $\dstab(I(G)) = \sum_{i=1}^p \dstab(I(G_i))-p+1$.
\end{enumerate}
}

\medskip

The second one estimates an upper bound for $\dstab(I(G))$. Before stating our result, we recall some terminologies from graph theory. In a graph $G$, a {\it leaf} is a vertex of degree one and a {\it leaf edge} is an edge incident with a leaf. A connected graph is called a {\it tree} if it contains no cycles, and it is called a {\it unicyclic} graph if it contains exactly one cycle. We use the symbols $\upsilon(G)$, $\varepsilon(G)$ and $\varepsilon_0(G)$ to denote the number of vertices, edges and leaf edges of $G$,  respectively.

\medskip

\noindent {\bf Theorem $\ref{MT}$}. {\it Let $G$ be a graph. Let $G_1,\ldots,G_s$ be all connected bipartite components of $G$ and let $G_{s+1},\ldots,G_{s+t}$ be all connected nonbipartite components of $G$. Let $2k_i$ be the maximum length of cycles of $G_i$ ($k_i :=1$ if $G_i$ is a tree) for all $i=1,\ldots,s$; and let $2k_i-1$ be the maximum length of odd cycles of $G_i$ for every $i = s+1,\ldots, s+t$. Then
$$\dstab(I(G)) \leqslant \upsilon(G) - \varepsilon_0(G) -\sum_{i=1}^{s+t} k_i +1.$$
}
\medskip

It is interesting that this bound is always achieved if $G$ has no cycles of length $4$ and every its connected component is either a tree or a unicyclic graph (see Theorem \ref{DMP}).

\medskip

Our approach is based on a generalized Hochster formula for computing local cohomology modules of arbitrary monomial ideals formulated by Takayama \cite{T}. The efficiency of this formula was shown in recent papers (see \cite{GH}, \cite{HT}, \cite{MT1}, \cite{MT2}, \cite{TT}). Using this formula and an explicit description of it for symbolic powers of Stanley-Reisner ideals given in \cite{MT1}, we are able to study the stability of depths of powers of edge ideals.

The paper is organized as follows. In Section $1$, we give some useful formulas on $\dstab(I(G))$ for the case when all components of $G$ are either nonbipartite or bipartite. We also recall the generalized Hochster formula to compute local cohomological modules of monomial ideals formulated by Takayama. In Section $2$  and Section $3$ we set up an upper bound of the index of depth stability for connected graphs which are either nonbipartite or bipartite, respectively. The core of the paper is Section $4$. There we compute the limit of the sequence $\depth R/I(G)^n$. Then combining with results in Sections $2$ and $3$ on the index of depth stability of connected graphs we obtain a bound of $\dstab(I(G))$ for all any graph $G$. In the last section, we compute  the index of depth stability of trees and unicyclic graphs.

\section{Preliminary}

We recall some standard notation and terminology from graph theory here. Let $G$ be a graph.  The ends of an edge of $G$ are said to be incident with the edge, and vice versa. Two vertices which are incident with a common edge are adjacent, and two distinct adjacent vertices are neighbors. The set of neighbors of a vertex $v$ in $G$ is denoted by $N_G(v)$ and the degree of a vertex $v$ in $G$, denoted by $\deg_G(v)$, is  the number of neighbours of $v$ in $G$. If there is no ambiguity in the context, we write $\deg v$ instead of $\deg_G(v)$. The graph $G$ is bipartite if its vertex set can be partitioned into two subsets $X$ and $Y$ so that every edge has one end in $X$ and one end in $Y$; such a partition $(X, Y )$ is called a bipartition of $G$. It is well-known that $G$ is bipartite if and only if $G$ contains no odd cycle (see \cite[Theorem $4.7$]{BM}).

Let $I$ be a homogeneous ideal in a polynomial ring $R=K[x_1,\ldots,x_r]$ over the field $K$. As introduced in \cite{HRV} we define {\it the index of depth stability} of $I$ to be the number
$$\dstab(I) := \min\{n_0\geqslant 1 \mid \depth S/I^n = \depth S/I^{n_0} \text{ for all } n\geqslant n_0\}.$$

In this paper we will establish a bound of $\dstab(I(G))$ for any graph $G$. First we have some information about $\dstab(I(G)))$ when every component of $G$ is nonbipartite.

\begin{lem} \label{LT21} Let $G$ be a graph with connected components $G_1,\ldots, G_t$. If all these components are nonbipartite, then
\begin{enumerate}
\item $\dstab(I(G)) = \min\{n\geqslant 1 \mid \depth R/I(G)^n = 0\}$;
\smallskip
\item $\dstab(I(G)) = \sum_{i=1}^t \dstab(I(G_i))-t+1$.
\end{enumerate}
\end{lem}

\begin{proof} $(1)$ Let $\mi_i := (x_j\mid j\in \ V(G_i))$ and $R_i := K[x_j \mid j \in V(G_i)]$, i.e.,   $\mi_i$ is the maximal homogeneous ideal of $R_i$, for $i=1,\ldots,t$. Let $\mi := (x_j \mid j\in V(G))$ be the maximal homogeneous ideal of $R$, so that $\mi =\mi_1 + \cdots + \mi_t$.

By \cite[Corollary $3.4$]{CMS} we have $\mi_i\in\ass(R_i/I(G_i)^{n_i})$ for some integer $n_i\geqslant 1$. Let $n_0 := \sum_{i=1}^t (n_i-1)+1$. By \cite[Corollary $2.2$]{CMS}  we have $\mi \in \ass(R/I(G)^n)$ for all $n\geqslant n_0$. On the other hand, the sequence $\{\ass(R/I(G)^n)\}_{n\geqslant 1}$ is increasing by \cite[Theorem $2.15$]{MMV} and note that $\depth R/I(G)^n = 0$ if and only if $\mi\in\ass(R/I(G)^n)$, this implies $\dstab(I(G)) = \min\{n\geqslant 1 \mid \depth R/I(G)^n = 0\}$.

$(2)$ By Part $1$ we also have $\dstab(I(G_i)) = \min\{n\geqslant 1 \mid \mi_i \in \ass(R/I(G_i)^n)\}$ for each component $G_i$. On the other hand, by \cite[Corollary $2.2$]{CMS} we have $\mi \in \ass(R/I(G)^n)$ if and only if we can write $n = \sum_{i=1}^t (n_i-1)+1$ where the $n_i$ are positive integers such that $\mi_i\in \ass(R_i/I(G_i)^{n_i})$. Thus the the statement follows.
\end{proof}

Next, we consider bipartite graphs. Note that all connected components of such graphs are bipartite as well. Bipartite graphs  have a nice algebraic characterization.

\begin{lem}\label{L03} {\rm (\cite{SVV})} A graph $G$ is bipartite if and only if  $I(G)^n = I(G)^{(n)}$ for all $n \geqslant 1$.
\end{lem}

Using this characterization we obtain.

\begin{lem}\label{C1} Let $G$ be a bipartite graph with $s$ connected components. Then
\begin{enumerate}
\item $\min\{\depth R/I(G)^n \mid n\geqslant 1\} = s$, and
\smallskip
\item $\dstab(I(G)) = \min\{n \geqslant 1 \mid \depth R/I(G)^n = s\}$.
\end{enumerate}
\end{lem}
\begin{proof} Since $G$ is bipartite, by Lemma $\ref{L03}$ we have $I(G)$ is normally torsion-free, and so by \cite{HO} the Rees ring  $\mathcal R[I(G)]$ of $I(G)$  is Cohen-Macaulay. Then by \cite{HU} the associated graded ring of $I(G)$  is Cohen-Macaulay as well. Hence, by \cite[Proposition $3.3$]{EH} we have
\begin{enumerate}
\item $\min\{\depth R/I(G)^n \mid n\geqslant 1\} = r -\ell(I(G))$, and
\smallskip
\item $\dstab(I(G)) = \min\{n \geqslant 1 \mid \depth R/I(G)^n = r -\ell(I(G))\}$.
\end{enumerate}
On the other hand, $r-\ell(I(G)) = s$ (see \cite[Page $50$]{W}). Thus the lemma follows.
\end{proof}

In the general case, our main tool to study $\dstab(I(G))$ is a generalized version of a Hochster's formula (see \cite[Theorem $4.1$ in Chapter II]{ST}) to compute local cohomology modules of monomial ideals given in \cite{T}.

Let $\mi := (x_1,\ldots, x_r)$ be the maximal homogeneous ideal of $R$ and $I$ a monomial ideal in $R$. Since $R/I$ is an $\N^r$-graded algebra, $H_{\mi}^i(R/I)$ is an $\Z^r$-graded module over $R/I$.  For every degree $\alb\in\Z^r$ we denote by $H_{\mi}^i(R/I)_{\alb}$ the $\alb$-component of $H_{\mi}^i(R/I)$.

Let $\Delta(I)$ denote the simplicial complex corresponding to the Stanley-Reisner ideal $\sqrt I$. For every $\alb = (\al_1,\ldots,\al_r) \in \Z^r$ we set $G_{\alb} := \{i \ | \ \al_i < 0\}$ and we denote by $\Delta_{\alb}(I)$ the simplicial complex of all sets of the form $F \setminus G_{\alb}$, where $F$ is a face of $\Delta(I)$ containing $G_{\alb}$ such that for every minimal generator $x^{\btb}$ of $I$ there exists an $i \notin F$ such that $\al_i < \bt_i$. To represent $\Delta_{\alb}(I)$ in a more compact way, for every subset $F$ of $[r]$ let $R_F := R[x_i^{-1} \ | \ i \in F \cup G_{\alb}]$ and $I_F := IR_F$. This means that the ideal $I_F$ of $R_F$ is generated by all monomials of $I$ by setting $x_i = 1$ for all $i \in F \cup G_{\alb}$. Then $x^{\alb} \in R_F$ and by \cite[Lemma $1.1$]{GH} we have
\begin{equation} \label{EQ01}  \Delta_{\alb}(I) = \{F\subseteq [r] \setminus G_{\alb} \ | \ x^{\alb} \notin I_F\}.\end{equation}

\begin{lem}\label{L01} {\rm (\cite[Theorem $1$]{T})} $\dim_K H_{\mi}^i(R/I)_{\alb} = \dim_K \widetilde{H}_{i-|G_{\alb}|-1}(\Delta_{\alb}(I); K).$
\end{lem}

Let $\mathcal F(\Delta)$ denote the set of facets of $\Delta$. If $\mathcal F(\Delta) = \{F_1, \ldots, F_m\}$, we write $\Delta = \left <F_1, \ldots, F_m\right >$. The Stanley-Reisner ideal of $\Delta$ can be written as (see \cite[Theorem $1.7$]{MS}):
$$I_{\Delta} = \bigcap_{F\in \mathcal F(\Delta)} P_F,$$
where $P_F$ is the prime ideal of $R$ generated by variables $x_i$ with  $i\notin F$. For every
integer $n \geqslant 1$, the $n$-th symbolic power of $I_{\Delta}$ is the monomial ideal
$$I_{\Delta}^{(n)} = \bigcap_{F\in \mathcal F(\Delta)} P_F^n.$$
Note that $\Delta(I_{\Delta}^{(n)}) = \Delta$. In \cite[Lemma $1.3$]{MT1} there was given an useful formula for computing $\Delta_{\alb}(I_{\Delta}^{(n)})$. We apply it to edge ideals.

An independent set in a graph $G$ is a set of vertices no two of which are adjacent to each other. An independent set $S$ in $G$ is maximal if the addition to $S$ of any other vertex in the graph destroys the independence. Let $\Delta(G)$ be the set of independent sets of $G$. Then $\Delta(G)$ is a simplicial complex and this complex is the so-called independence complex of $G$; and facets of $\Delta(G)$ are just maximal independent sets of $G$. It is easy to see that $I(G) = I_{\Delta(G)}$.

Now we can compute $\Delta_{\alb}(I(G)^n)$ for bipartite graphs $G$.

\begin{lem}\label{L02} Let $G$ be a bipartite graph. Then, for all $\alb\in \N^r$ and $n \geqslant 1$, we have
$$\Delta_{\alb}(I(G)^n) = \left < F\in \mathcal F(\Delta(G)) \ | \ \sum_{i\notin F} \al_i \leqslant n-1\right>.$$
\end{lem}

\begin{proof} Let $\Delta :=\Delta(G)$. Then, $I_{\Delta} = I(G)$. By Lemma $\ref{L03}$, we have $I(G)^n =I(G)^{(n)}$. Therefore, $\Delta_{\alb}(I(G)^n) = \Delta_{\alb}(I_{\Delta}^{(n)})$. The lemma now follows from \cite[Lemma $1.3$]{MT1}.
\end{proof}

We conclude this section with some remarks about operations on monomial ideals. Let $A := K[x_1,\ldots,x_s], B:= K[y_1,\ldots, y_t]$ and $R := K[x_1,\ldots, x_s, y_1,\ldots, y_t]$ be polynomial rings where $\{x_1,\ldots, x_s\}$ and $\{y_1,\ldots, y_t\}$ are two disjoint sets of variables. Then for monomial ideals $I,I_1,I_2$ of $R$ we have
\begin{equation} \label{M1} I\cap (I_1+I_2) = I\cap I_1 + I \cap I_2.
\end{equation}
Let $I_1,I_2$ be monomial ideals in $A$ and let $J_1,J_2$ be monomial ideals in $B$. For simplicity, we denote $I_sR$ by $I_s$ and $J_sR$ by $J_s$ for $s=1,2$, then by \cite[Lemma $1.1$]{HT1} we have
\begin{equation} \label{M2} I_1 J_1 \cap I_2 J_2 = (I_1 \cap I_2) (J_1 \cap J_2).
\end{equation}

\begin{lem}\label{L10} Let $I$ be a proper monomial ideal of $A$ and $J$ a proper monomial ideal of $B$. Then, for all $n\geqslant 1$ we have
$$\depth R/(I + J)^n \geqslant \min\{\depth A/I^m \mid 1\leqslant m \leqslant n\}.$$
\end{lem}
\begin{proof} Since the case $I=\zv$ or $J =\zv$ is obvious, so we may assume that $I$ and $J$ are nonzero ideals. For each $i=0,\ldots,n$, we put:
$$W_i := I^iJ^{n-i} +\cdots+I^nJ^0 \subseteq R,$$
where $I^0 = J^0 = R$. Since $W_0 = (I+J)^n$, in order to prove the lemma it suffices to show that
\begin{equation}\label{Wi}
\depth R/W_i \geqslant  \min\{\depth A/I^j\mid \max\{i,1\} \leqslant j \leqslant n\} \text{ for all } i=0,\ldots,n.
\end{equation}
Indeed, if $i=n$, then $\depth R/W_n =\depth R/I^n = \depth A/I^n +t \geqslant \depth A/I^n$. Next assume that the claim holds for $i+1$ with $0\leqslant i < n$. By Equations ($\ref{M1}$) and ($\ref{M2}$) we have $I^iJ^{n-i}\cap W_{i+1} = I^{i+1}J^{n-i}$. Since $W_i = I^iJ^{n-i} +W_{i+1}$, we have  an exact sequence
$$\zv \longrightarrow R/I^{i+1}J^{n-i} \longrightarrow R/I^iJ^{n-i} \oplus R/W_{i+1} \longrightarrow R/W_i \longrightarrow \zv.$$
By Depth Lemma (see, e.g.,  \cite[Proposition $1.2.9$]{BH}), we have
$$\depth R/W_i \geqslant \min\{\depth R/I^{i+1}J^{n-i}-1,  \depth R/I^iJ^{n-i}, \depth R/W_{i+1}\}.$$
On the other hand, by \cite[Lemma $2.2$]{HT1} we have
$$\depth R/I^{i+1}J^{n-i}-1 = \depth A/I^{i+1} + \depth B/J^{n-i} \geqslant  \depth A/I^{i+1}.$$
Together with the induction hypothesis we then get
$$\depth R/W_i \geqslant \min\{\depth R/I^iJ^{n-i}, \depth A/I^j \mid j=i+1,\ldots,n\}.$$

If $i \geqslant 1$, by \cite[Lemma $2.2$]{HT1} we have $$\depth R/I^iJ^{n-i} = \depth A/I^i + \depth B/J^{n-i}+1 \geqslant \depth A/I^i,$$
which yields the claim.

If $i = 0$, then $\depth R/W_0 \geqslant \min\{\depth R/J^{n}, \depth A/I^j \mid j=1,\ldots,n\}$. Note that
$\depth R/J^{n} = s +\depth B/J^n \geqslant s \geqslant \depth A/I$, hence the claim also holds. The proof now is complete.
\end{proof}

\section{Depths of powers of edge ideals of connected nonbipartite graphs}

Note that for a graph $G$ we always assume that $V(G) = [r]$; $R=K[x_1,\ldots,x_r]$ is a polynomial ring over fields $K$ and $\mi =(x_1,\ldots,m_r)$ is the maximal homogeneous ideal of $R$. In this section we always assume that $G$ is a connected nonbipartite graph.

By Lemma $\ref{LT21}$ we have $\dstab(I(G)) = \min\{n\geqslant 1 \mid \mi \in \ass R/I(G)^n\}$. Based on \cite{CMS}, we will determine explicitly when $\mi\in \ass R/I(G)^n$ for a unicylic graph $G$.

Recall that a vertex cover (or a cover) of  $G$ is a subset $S$ of $V(G)$ such that every edge of $G$ has at least one endpoint in $S$. A cover is minimal if none of its proper subsets is itself a cover. It is well-known that $P = (x_{i_1}, \ldots ,x_{i_t})$ is a minimal prime of the edge ideal $I(G)$ if and only if $\{i_1, \ldots , i_t\}$ is a minimal cover of $G$. For a subset $U$ of $V(G)$, the neighbor set of $U$ is the set
$$N(U) := \{v \in V(G) \mid v \text{ is adjacent to some vertex in } U\}.$$

We now describe the process that builds $\ass R/I(G)^n$ for a unicylic graph $G$. Let $C$ be a cycle of $G$ of length $2k-1$. Let $R_k$ be the set of vertices of $C$, $B_k:=N(R_k)\setminus R_k$ and a monomial
$$d_k := \prod_{i\in R_k} x_i.$$

We now build recursively sets $R_n$, $B_n$ and a monomial $d_n$ for $n\geqslant k$. Suppose that $i\in R_s$ and $j\in R_s\cup B_s$ for some $s\geqslant k$ such that $\{i,j\}$ is an edge of $G$. Now if $j\in R_s$, then let $R_{s+1}:=R_s$ and $B_{s+1}:=B_s$. If $j\in B_s$, then let $R_{s+1} := R_s\cup \{j\}$ and $B_{s+1} := (B_s \cup N(j))\setminus R_{s+1}$. In either case, let $d_{s+1} := d_s (x_ix_j)$.

Now for such a couple $(R_n,B_n$) with $n\geqslant k$, we take $V$ to be any minimal subset of $V(G)$ such that $R_n\cup B_n\cup V$ is a  cover of $G$. Then, $(R_n,B_n,V) := (x_i\mid i\in R_n\cup B_n\cup V)$ is an associated prime of $R/I(G)^n$ by \cite[Theorem $3.3$]{CMS}. Let $P_n$ be the set of  such all prime ideals. Then, 
by \cite[Theorem $5.6$]{CMS} we have
\begin{equation}\label{ASSOC}
\ass R/I(G)^n=\Min(R/I(G))\cup P_n.
\end{equation}

For unicyclic graphs, we have the following observation.

\begin{rem} \label{rem_path} Assume that $G$ is a unicyclic graph with a cycle $C$ such that $G\ne C$. For any $v\in V(G)\setminus V(C)$, there is a unique simple path of the form: $v_0,v_1\ldots,v_d$, where $v_0\in V(C)$, $v_1,\ldots,v_d\notin V(C)$ and $v_d = v$. We say that this path connects $C$ and $v$.  Moreover,
\begin{enumerate}
\item $d_G(v,C) = d$.
\item This simple path can extend to a simple path connecting $C$ to a leaf, i.e., there are vertices $u_1,\ldots, u_t$ such that $u_s$ is a leaf and $v_0,v_1\ldots,v_d,u_1,\ldots, u_t$ is a simple path.
\item If $d_G(v,C)$ is maximal, i.e., $d_G(v,C) \geqslant d_G(u,C)$ for any $u\in V(G)$, then $v$ is a leaf.  Assume further that $d\geqslant 2$, then $N_G(v_{d-1})$ contains only one non-leaf $v_{d-2}$.
\end{enumerate}
\end{rem}

We now can determine $\dstab(I(G))$ with unicyclic nonbipartite graphs $G$.

\begin{lem}\label{T6} Let $G$ be a unicyclic nonbipartite graph. If the length of the unique cycle is $2k-1$, then $\dstab(I(G)) = \upsilon(G)-\varepsilon_0(G)-k+1$.
\end{lem}
\begin{proof} By \cite[Corollaries $3.4$ and $4.3$]{CMS} we have 
$$\mi\in\ass R/I(G)^n \text{ for all } n\geqslant \upsilon(G)-\varepsilon_0(G)-k+1.$$
Therefore,
$$\depth R/I(G)^n = 0 \text{ for all } n\geqslant \upsilon(G)-\varepsilon_0(G)-k+1,$$
so that $\dstab (I(G)) \leqslant \upsilon(G)-\varepsilon_0(G)-k+1$.

We next prove the converse inequality. It suffices to show that if $\mi\in \ass R/I(G)^n$, then $n\geqslant  \upsilon(G)-\varepsilon_0(G)-k+1$.

By Equation $(\ref{ASSOC})$ we deduce that $\mi\in P_n$. Thus, $\mi = (R_n,B_n,V)$ where $V$ is a minimal subset of $V(G)$ such that $R_n\cup B_n\cup V$ is a vertex cover of $G$. In particular, $V(G)=R_n\cup B_n\cup V$. 

{\it Claim $1$:} $V=\emptyset$. Indeed,  if $V$ contains no leaves of $G$, then every leaf of $G$ is in either $R_n$ or $B_n$, and so $R_n\cup B_n=V(G)$ by Remark $\ref{rem_path}$. This forces $V=\emptyset$.

Suppose $V$ contains a leaf, say $i$. Let $j$ be the unique neighbor of $i$ in $G$. Then, $j\in V(G)=R_n\cup B_n\cup V$. Therefore, $R_n\cup B_n\cup (V\setminus \{i\})$ is also a vertex cover of $G$. This contradicts the minimality of $V$. Hence,  $V=\emptyset$, as claimed.

{\it Claim $2$:} $|B_n|\leqslant \varepsilon_0(G)$. Indeed, assume on the contrary that $|B_n| >|\varepsilon_0(G)|$, so that $B_n$ contains a non-leaf of $G$, say $i$. Let $p$ be a simple path connecting $C$ and a leaf that passes through $i$. Let $j$ be a vertex of $p$ after $i$. Then, by Remark $\ref{rem_path}$ and the construction of $R_n$ and $B_n$ we deduce that $j\notin R_n\cup B_n$, so $j\notin V(G)$ by Claim $1$, a contradiction. Hence, $|B_n| \leqslant \varepsilon_0(G)|$, as claimed.

We now prove the lemma. Since $|R_k|=2k-1$ and $|R_n|\leqslant |R_k| + (n-k)$, together with Claim $2$ we obtain $\upsilon(G) = |R_n| + |B_n| \leqslant |R_k|+(n-k)+\varepsilon_0(G) = n+k-1+\varepsilon_0(G),$
so $n\geqslant \upsilon(G)-\varepsilon_0(G)-k+1$, as required.
\end{proof}

\begin{lem} \label{A1} Let $G$ be a unicyclic nonbipartite graph. Assume that the unique odd cycle of $G$ is of length ${2k-1}$. Let $n:=\upsilon(G)-\varepsilon_0(G)-k+1$. Then, there is a monomial $f$ of $R$  such that
$\deg f  = 2n-1$ and $f x_i \in I(G)^n$ for all $i=1,\ldots,r$.
\end{lem}

\begin{proof} By Lemma $\ref{T6}$ and Equation $(\ref{ASSOC})$ we have $\mi\in P_n$. Thus, $\mi = (R_n,B_n,V)$ where $V$ is a minimal subset of $V(G)$ such that $R_n\cup B_n\cup V$ is a vertex cover of $G$. In particular, $V(G)=R_n\cup B_n\cup V$. By the same way as in the proof of Claim $1$ in Lemma $\ref{T6}$ we have $V = \emptyset$. Hence, $R_n\cup B_n = \{1,\ldots,r\}$.

Let $f := d_n$. Together with \cite[Lemma $3.2$]{CMS} we imply that $\deg (f) = 2n-1$ and $f x_i \in I(G)^n$ for all $i=1,\ldots,r$, as required.
\end{proof}

Let $G$ be a connected nonbipartite graph and let $2l-1$ be the minimum length of odd cycles of $G$. Then $\dstab(G) \leqslant \upsilon(G)-\varepsilon_0(G)-l+1$ by \cite[Corollaries $3.4$ and $4.3$]{CMS}. The following result improves this bound a little bit.

\begin{prop}\label{A3} Let $G$ be a connected nonbipartite graph. Let $2k-1$ be the maximum length of odd cycles of $G$. Then, $\dstab(I(G)) \leqslant \upsilon(G) -\varepsilon_0(G) - k+1$.
\end{prop}
\begin{proof} Let $C$ be an odd cycle of $G$ of length $2k-1$. If $C'$ is another cycle of $G$, then $C'$ has an edge $e$  not lying on the cycle $C$. Delete this edge from $G$, thereby obtaining a connected subgraph $G'$ of $G$ with $V(G') = V(G)$ and $C$ is a cycle of $G'$. This process continues until we obtain a connected subgraph $H$ of $G$ such that $V(G)=V(H)$ and $H$ has only one cycle, that is $C$. Let $n :=\upsilon(H)-\varepsilon_0(H)-k+1$. By Lemma \ref{A1}, there is a monomial $f\in R$ such that $\deg f = 2n-1$ and
$x_if\in I(H)^n$ for all $i=1,\ldots,r$. Since $I(H)\subseteq I(G)$, we have
\begin{equation}\label{A4} x_if\in I(G)^n \text{ for all } i=1,\ldots,r.\end{equation}
As $I(G)$ is generated by quadratic monomials and $\deg f =2n-1$, so $f\notin I(G)^n$. Together with Equation ($\ref{A4}$) one has
$I(G)^n:f =\mi$. Hence, $\depth R/I(G)^n = 0$, which implies $\dstab(I(G)) \leqslant n$ by Lemma \ref{LT21}. Since $\upsilon(G)=\upsilon(H)$ and $\varepsilon_0(G) \leqslant \varepsilon_0(H)$,
$$\dstab(I(G)) \leqslant n \leqslant \upsilon(G) -\varepsilon_0(G)-k+1,$$
as required.
\end{proof}

\section{Depths of powers of edge ideals of connected bipartite graphs}

Let $G$ be a biparite graph with bipartition $(X,Y)$. Clearly, $X$ and $Y$ are then facets of $\Delta(G)$. Assume further that $G$ is connected. By Lemma $\ref{C1}$, one has $\dstab(I(G)$ is the smallest integer $n$ such that $\depth R/I(G)^n = 1$. For such graphs we can find $\dstab(I(G))$ via integer linear programming.

\begin{lem}\label{LT1} Let $G$ be a connected bipartite graph with bipartition $(X,Y)$ and $n$ a positive integer. Then, $\depth R/I(G)^n = 1$ if and only if $\Delta_{\alb}(I(G)^n)=\left<X,Y\right>$ for some $\alb=(\al_1,\ldots,\al_r)\in \N^r$. Moreover, if $n=\dstab(I(G))$, then such $\alb$ must satisfy
$$\sum_{i\notin X}\al_i = \sum_{i\notin Y}\al_i = n-1.$$
\end{lem}

\begin{proof} Since $G$ is bipartite, by Lemma $\ref{L03}$ one has $I(G)^n = I(G)^{(n)}$. Hence,
$$\depth R/I(G)^n = \depth R/I(G)^{(n)} \geqslant 1,$$
and hence $\depth R/I(G)^n = 1$ if and only if $H_{\mi}^1(R/I(G)^n) \ne \zv$. By \cite [Corollary $1.2$]{MT1} this is equivalent to the condition $\Delta_{\alb}(I(G)^n)$ being disconnected for some $\alb=(\al_1,\ldots,\al_r) \in \N^r$.

Therefore, in order to prove the lemma it suffices to show that if $\Delta_{\alb}(I(G)^n)$ is disconnected, then $\Delta_{\alb}(I(G)^n)=\left<X,Y\right>$. Indeed, since $\Delta_{\alb}(I(G)^n)$ is disconnected, there are two facets $F$ and $H$ of it such that $F \cap H =\emptyset$. Hence, $(V(G) \setminus F) \cup (V(G) \setminus H) = V(G)$. Together with the fact that $X \cap Y = \emptyset$ and $X \cup Y = V(G)$ we get
$$\sum_{i \notin X}\al_i +\sum_{i \notin Y}\al_i = \sum_{i \in V(G)}\al_i \leqslant \sum_{i \notin F}\al_i +\sum_{i \notin H}\al_i.$$
Since $F$ and $H$ are members of $\mathcal F(\Delta_{\alb}(I(G)^n))$, by Lemma $\ref{L02}$ we have
$$\sum_{i \notin F}\al_i\leqslant n-1, \ \text{ and } \sum_{i \notin H}\al_i \leqslant n-1.$$
Therefore,
$$\sum_{i \notin X}\al_i +\sum_{i \notin Y}\al_i \leqslant 2(n-1),$$
which yields $$\sum_{i \notin X}\al_i\leqslant n-1 \ \text{ or } \sum_{i \notin Y}\al_i\leqslant n-1.$$
Thus we may assume that $$\sum_{i \notin X}\al_i\leqslant n-1,$$
and thus $X \in \Delta_{\alb}(I(G)^n)$ by Lemma $\ref{L02}$.  As $\Delta_{\alb}(I(G)^n)$ is disconnected, there
is a facet $L$ of $\Delta_{\alb}(I(G)^n)$ such that $X \cap L =\emptyset$. We then have $L\subseteq V(G) \setminus X = Y$. The maximality of $L$ forces $L = Y$, hence $Y\in \Delta_{\alb}(I(G)^n)$. If $\Delta_{\alb}(I(G)^n)$ has another facet, say $T$, that is different from $X$ and $Y$, then neither $X$ nor $Y$ contains $T$, and then $T$ meets both $X$ and $Y$. This is impossible since $\Delta_{\alb}(I(G)^n)$ is disconnected. Hence, $\Delta_{\alb}(I(G)^n) = \left<X, Y\right>$, as claimed.

Finally, assume that $n=\dstab(I(G))$. Then, by Lemma $\ref{C1}$, $n$ is the smallest positive integer such that $\depth R/I(G)^n=1$.

Assume that $\sum_{i\notin X}\al_i < n-1$ and $\sum_{i\notin Y}\al_i < n-1$. Then, $n-1\geqslant 1$ and
$$\sum_{i\notin X}\al_i \leqslant  (n-1)-1 \text{ and } \sum_{i\notin Y}\al_i \leqslant (n-1)-1.$$
If $F$ is a facet of $\Delta(G)$ that is different from $X$ and $Y$, then $F\notin \mathcal F(\Delta_{\alb}(I(G)))$, and then $\sum_{i\notin F} \al_i\geqslant n  > n-1$ according to Lemma $\ref{L02}$. From these equations and Lemma $\ref{L02}$, we get $\Delta_{\alb}(I(G)^{n-1}) = \left<X, Y\right>$. In particular, $\Delta_{\alb}(I(G)^{n-1})$ is disconnected, so  $\depth R/I(G)^{n-1}=1$. This contradicts to the minimality of $n$. Thus, we may assume that $\sum_{i\notin Y}\al_i = n-1$.

Assume now that $\sum_{i\notin X}\al_i < n-1$. Since $$\sum_{i\in X}\al_i =\sum_{i\notin Y}\al_i = n-1 \geqslant 1,$$
$\al_i \geqslant 1$ for some $i\in X$. We may assume that $i=1$. Let $\btb = (\al_1-1,\al_2,\ldots,\al_r)$, so that $\btb\in\N^r$ as $\al_1\geqslant 1$. By the same way as in the previous paragraph we get $\Delta_{\btb}(I(G)^{n-1})=\left<X,Y\right>$, which yields $\depth R/I(G)^{n-1}=1$. This also contradicts to the minimality of $n$. Hence,
$$\sum_{i\notin X} \al_i = \sum_{i \notin Y} \al_i = n-1,$$
as required.
\end{proof}

We now give an explicit solution of the equation $\Delta_{\alb}(I(G)^n)=\left<X,Y\right>$. This solution turns out to be optimal for studying $\dstab(I(G))$.

\begin{defn}Let $G$ be a graph. We define:
\begin{enumerate}
\item For each $i \in V(G)$, denote $\mu_G(i)$ to be the number of non-leaf edges of $G$ that are incident with $i$,
\item $\mu(G) := (\mu_G(1),\ldots, \mu_G(r)) \in \N^r$.
\end{enumerate}
\end{defn}

\begin{lem}\label{L06} Let $G$ be a connected bipartite graph with bipartition $(X, Y)$. Let $\alb := \mu(G)$ and $n := \varepsilon(G) -\varepsilon_0(G) + 1$. Then,
$$\Delta_{\alb}(I(G)^n) = \left<X,Y\right>, \ \text{ and } \sum_{i\notin X}\al_i = \sum_{i\notin Y}\al_i = \varepsilon(G) -\varepsilon_0(G).$$
\end{lem}
\begin{proof}Clearly, $X$ and $Y$ are facets of $\Delta(G)$. If $\upsilon(G) = 2$, i.e., $G$ is exactly an edge $\{1,2\}$, then $n = 1$ and $\alb = (0, 0)$. We may assume that $X = \{1\}$ and $Y = \{2\}$. Then, $\Delta_{\alb}(I(G)^n) = \Delta(I(G))=\Delta(G)=\left<\{1\},\{2\}\right>$, so the lemma holds for this case.

Assume that $\upsilon(G) \geqslant 3$. Let $S := \{i \in X \ | \ \deg i = 1\}$ and $T := \{j \in Y \ | \ \deg j = 1\}$, so that
\begin{equation}\label{EQ06} |S| + |T| = \varepsilon_0(G).\end{equation}
From \cite[Theorem $1.1$ and Exercise $1.1.9$]{BM}  we have
\begin{equation}\label{EQ07}\sum_{i\in X}\deg i = \sum_{j\in Y}\deg j = \varepsilon(G).\end{equation}
Note that the unique neighbor of each leaf of $G$ in $X$ is a non-leaf of $G$ in $Y$. Together with Formulas ($\ref{EQ06}$)-($\ref{EQ07}$), this fact gives
$$\sum_{i\in X}\mu_G(i) = \sum_{i\in X}\deg i - |S| - |T| = \varepsilon(G)-\varepsilon_0(G) = n-1.$$
Similarly,
$$\sum_{j\in Y}\mu_G(j) = \sum_{j\in Y}\deg j - |S| - |T| = \varepsilon(G)-\varepsilon_0(G) = n-1.$$
Hence, $X,Y \in \mathcal F(\Delta_{\alb}(I(G)^n))$ by Lemma $\ref{L02}$. So in order to prove the lemma it remains to prove that $\Delta_{\alb}(I(G)^n) = \left<X,Y\right>$, or equivalently, if $F\in \mathcal F(\Delta(G)) \setminus \{X,Y\}$ then $F\notin F(\Delta_{\alb}(I(G)^n))$.

Indeed, by the maximality of $F$, we can partition $F$ into $F = U \cup V$, where $U$ and $V$ are nonempty proper subsets of $X$ and $Y$, respectively, such that every vertex in $X \setminus U$ (resp. in $Y\setminus V$) is adjacent to at least one vertex in $V$ (resp. in $U$), and no vertex in $U$ is adjacent to a vertex in $V$. Then, we have
$$\sum_{i\in X\setminus U} \mu_G(i) = \sum_{i\in X\setminus U} \deg i - |S \cap (X\setminus U)| - |T\cap V|,$$
$$\sum_{j\in Y\setminus V} \mu_G(j) = \sum_{j\in Y\setminus V} \deg j - |T \cap (Y\setminus V)| - |S\cap U|,$$
and
$$\sum_{j\in Y\setminus V} \deg j = \sum_{i\in U} \deg i +\sum_{j\in Y\setminus V} |N_G(j) \cap (X\setminus U)|.$$
Combining these Equations with Formulas ($\ref{EQ06}$)-($\ref{EQ07}$) we obtain
\begin{align*}
&\sum_{u\notin F} \mu_G(u) = \sum_{i\in X\setminus U}\mu_G(i) +\sum_{j\in Y\setminus V} \mu_G(j)\\
&=\sum_{i\in X\setminus U} \deg i - |S\cap (X\setminus U)| - |T\cap V| +\sum_{j\in Y \setminus V} \deg j - |T\cap(Y\setminus V)| - |S\cap U|\\
&=\sum_{i\in X\setminus U} \deg i +\sum_{j\in Y \setminus V} \deg j - (|S\cap U|+|S\cap (X\setminus U)| + |T\cap V|  + |T\cap(Y\setminus V)|)\\
&=\sum_{i\in X\setminus U} \deg i +\sum_{j\in Y \setminus V} \deg j - (|S|+|T|)\\
&=\sum_{i\in X\setminus U} \deg i +\sum_{i\in U} \deg i + \sum_{j\in Y\setminus V} |N_G(j) \cap (X\setminus U)| - \varepsilon_0(G)\\
&=\sum_{i\in X} \deg i - \varepsilon_0(G) + \sum_{j\in Y\setminus V} |N_G(j) \cap (X\setminus U)| \\
&=\varepsilon(G) - \varepsilon_0(G) + \sum_{j\in Y\setminus V} |N_G(j) \cap (X\setminus U)|,
\end{align*}
or equivalently,
$$\sum_{u\notin F} \mu_G(u) = \varepsilon(G) - \varepsilon_0(G) + |P| = n -1 + |P|,$$
where $P = \{(a,b)\ | \ a \in X \setminus U, b\in Y\setminus V \ \text{ and } ab \in E(G)\}$. Therefore, by Lemma $\ref{L02}$ we have $F \notin \Delta_{\alb}(I(G)^n)$ whenever $|P| \geqslant 1$, i.e., $P \ne \emptyset$.

In order to prove $P \ne\emptyset$, let  $\ell := \min \{d_G(i, j) \ | \ i \in U \ \text{ and } j \in V\}$. Then, $\ell$ is finite because $G$ is connected. Let $a \in U$ and $b \in V$ such that there is a path of length $\ell$ connects $a$ and $b$. Suppose
$$a = a_1,b_1,a_2,b_2, \ldots, a_s, b_s = b$$
is such a path, where $a_1,\ldots, a_s \in X$ and $b_1,\ldots, b_s \in Y$ . Then, $b_1 \in Y\setminus V$ because $a_1 = a \in U$. Now if $a_2 \in U$, then we would have the path $a_2, b_2, \ldots, a_s, b_s = b$ that connects $a_2 \in U$ and $b \in V$ of length $\ell-2$. This contradicts to the minimality of $\ell$. Thus, $a_2 \in X \setminus U$. This implies $(a_2, b_1) \in P$, so $P \ne \emptyset$, as required.
\end{proof}

Let $G$ be a graph and $C$ be a cycle of $G$. For any vertex $v$ of $G$, we define the distance from $v$ to $C$ to be:
$$d_G(v,C) = \{d_G(v,u)\mid u\in V(C)\}.$$

\begin{prop}\label{L07} Let $G$ be a connected bipartite graph and let $2k$ be the maximum length of cycle of $G$ ($k := 1$ if $G$ is a tree). Then, $\dstab(I(G))\leqslant \upsilon(G) -\varepsilon_0(G)-k+1$.
\end{prop}
\begin{proof} Let $(X,Y)$ be a bipartition of $G$.

If $G$ is a tree, then $\varepsilon(G)=\upsilon(G)-1$ by \cite[Theorem $4.3$]{BM}. Let $\alb :=\mu(G)$ and $ n := \varepsilon(G) -\varepsilon_0(G) + 1$. Then, $\Delta_{\alb}(I(G)^n) = \left<X,Y\right>$ by Lemma $\ref{L06}$. Hence, by Lemma $\ref{LT1}$, we have $$\dstab(I(G))\leqslant n=\varepsilon(G) -\varepsilon_0(G) + 1 = \upsilon(G)-\varepsilon_0(G),$$ and the proposition follows.

Assume that $G$ has a cycle, say $C_{2k}$,  of length $2k$ where $k\geqslant 2$. If $C$ is another cycle of $G$, then $C$ has an edge $e$ not lying in the cycle $C_{2k}$. Delete this edge from $G$, thereby obtaining a connected subgraph $G'$ of $G$ with $V(G') = V(G)$ and $C_{2k}$ is a cycle of $G'$. This process continues until we obtain a connected subgraph $H$ of $G$ such that $V(G)=V(H)$ and $H$ has only one cycle, that is $C_{2k}$. Note that $H$ is also a bipartite graph with bipartition $(X,Y)$. Assume that the cycle $C_{2k}$ is: $$1,2,\ldots,2k-1,2k,1.$$
Let $n :=\upsilon(H)-\varepsilon_0(H)-k+1$ and define $\alb =(\al_1,\ldots,\al_r)\in \N^r$ by
$$\al_j :=  \left\{ \begin{array}{ll}
\mu_H(j)-1 &\mbox{if $1\leqslant j \leqslant 2k+2$},\\
\mu_H(j)  &\mbox{otherwise}.
\end{array}
\right.
$$

{\it Claim $1$}: $$\Delta_{\alb}(I(H)^n) =\left<X,Y\right> \text{ and } \sum_{i\notin X}\al_i = \sum_{i\notin Y}\al_i = n-1.$$

{\it Proof:} We will prove this claim by induction on $\upsilon(H)$. If $\upsilon(H) = 2k$, then $H=C_{2k}, r = 2k$ and $n = k+1$.  We may assume also that $X = \{1,3,\ldots,2k-1\}$ and $Y=\{2,4,\ldots,2k\}$. By noticing that $\alb=(1,1,\ldots,1)\in \N^r$, we have
$$\sum_{i\notin X}\al_i = \sum_{i\notin Y}\al_i = k = n-1,$$
and therefore $X$ and $Y$ are facets of $\Delta_{\alb}(I(H)^n)$. Hence, it remains to show that $\Delta_{\alb}(I(H)^n)=\left<X,Y\right>$. Let $F$ be a facet of $\Delta(H)$ that is different from $X$ and $Y$. Since all facets of $\Delta(C_{2k})$ have at most $k$ elements; and only $X$ and $Y$ have exactly $k$ elements, we must have $|F| < k$. Hence,
$$\sum_{i\notin F}\al_i \geqslant k + 1 = n,$$
and hence $F\notin\Delta_{\alb}(I(H)^n)$. Therefore, $\Delta_{\alb}(I(H)^n)=\left<X,Y\right>$, and the claim follows.

Assume that $\upsilon(H) > 2k$. Clearly, $r$ is not in $C_{2k}$, so we may assume that $d_G(r,C_{2k})\geqslant d_G(v,C_{2v})$ for any vertex $v$ of $G$. Then, $r$ is a leaf by Remark $\ref{rem_path}$. Let $t$ be the unique neighbor of $r$ in $G$.

Let $T := H \setminus \{r\}$. Then, $T$ is also a connected bipartite graph with only cycle $C_{2k}$ and $\upsilon(T)=\upsilon(H)-1$. We may assume that $r\in X$, so that $(X\setminus\{r\},Y)$ is a bipartition of $T$. Let $s:=\upsilon(T)-\varepsilon_0(T)-k+1$ and define $\btb =(\bt_1,\ldots,\bt_{r-1})\in \N^{r-1}$ by
$$\bt_j :=  \left\{ \begin{array}{ll}
\mu_T(j)-1 &\mbox{if $1\leqslant j \leqslant 2k$},\\
\mu_T(j)  &\mbox{otherwise}.
\end{array}
\right.
$$
We now distinguish two cases:

{\it Case $1$:} $d_G(r, C_{2k})=1$. In this case $V(G)\setminus V(C_{2k})$ is the set of all leaves of $G$. Thus, $\btb = (\al_1,\ldots, \al_{r-1})$ and $\varepsilon_0(T) = \varepsilon_0(H)-1$, and thus $s = n$.

Since $\upsilon(T)=\upsilon(H)-1$ and $\alpha_r = 0$, by the induction hypothesis we have $\Delta_{\btb}(I(T)^{n}) =\left<X\setminus\{r\},Y\right>$, and
\begin{equation}\label{V2}\sum_{i\notin X} \al_i =\sum_{i\notin X\setminus\{r\}} \bt_i = n-1, \text{ and } \sum_{i\notin Y} \al_i =\sum_{i\notin Y} \bt_i +\al_r = n-1. \end{equation}
In particular, $X\in\Delta_{\alb}(I(H)^n)$ and $Y\in\Delta_{\alb}(I(H)^n)$. Thus it remains to show that $\Delta_{\alb}(I(H)^n) =\left<X,Y\right>$. Let $F$ be any facet of $\Delta(H)$ that is different from $X$ and $Y$.

Assume that $t\in F$. Then, $F$ is also a facet of $\Delta(T)$ that is different from $X\setminus\{r\}$ and $Y$. Therefore, 
$$\sum_{i\notin F}\al_i = \sum_{i\notin F}\bt_i +\al_r = \sum_{i\notin F}\bt_i \geqslant  n.$$
Therefore, $F\notin \Delta_{\alb}(I(H)^n)$.

Assume that $t\notin F$. Then, $r\in F$ and  $F\setminus\{r\}$ is a subset of neither $X\setminus\{r\}$ nor $Y$. Since $F\setminus \{r\}\in \Delta(T)$, there is a facet $F'$ of $\Delta(T)$ containing $F$ and being different from $X\setminus\{r\}$ and $Y$. Therefore,
$$\sum_{i\notin F}\al_i = \sum_{i\notin F\setminus\{r\}}\bt_i \geqslant \sum_{i\notin F'}\bt_i \geqslant n.$$
Which implies $F\notin \Delta_{\alb}(I(H)^n)$. The claim holds for this case.

{\it Case $2$:} $d_G(r,C_{2k})\geqslant 2$. By Remark $\ref{rem_path}$ we can assume that $N_G(t) = \{t-1,t+1,\ldots,r\}$ where $t-1$ is a non-leaf and $t+1,\ldots,r$ are leaves. We now distinguish two subcases:

{\it Case $2a$:} $t + 1 = r$. Then, $\varepsilon_0(T) = \varepsilon_0(H)$ and $s=n-1$. Since $\upsilon(T)=\upsilon(H)-1, \al_r = 0$ and
$$\bt_j =  \left\{ \begin{array}{ll}
\al_j-1 &\mbox{if $j=t-1$ or $j = t$},\\
\al_j  &\mbox{otherwise},
\end{array}
\right.
$$
by the induction hypothesis we have $\Delta_{\btb}(I(T)^{n-1}) =\left<X\setminus\{r\},Y\right>$, and
\begin{equation}\label{V2}\sum_{i\notin X} \al_i =\sum_{i\notin X\setminus\{r\}} \bt_i+1 = n-1, \text{ and } \sum_{i\notin Y} \al_i =\sum_{i\notin Y} \bt_i +\al_r+1 = n-1. \end{equation}
In particular, $X\in\Delta_{\alb}(I(H)^n)$ and $Y\in\Delta_{\alb}(I(H)^n)$. Thus it remains to show that $\Delta_{\alb}(I(H)^n) =\left<X,Y\right>$. Let $F$ be any facet of $\Delta(H)$ that is different from $X$ and $Y$.

Assume that $t\in F$. Then, $F$ is also a facet of $\Delta(T)$ that is different from $X\setminus\{r\}$ and $Y$. Since $t-1\notin F$ and $\al_{t-1}=\bt_{t-1}+1$, we have
$$\sum_{i\notin F}\al_i = \sum_{i\notin F}\bt_i +1+\al_r = \sum_{i\notin F}\bt_i+1 \geqslant s+1 = n.$$
Therefore, $F\notin \Delta_{\alb}(I(H)^n)$.

Assume that $t\notin F$. Then, $r\in F$. If $t-1\in F$, then
$F\setminus\{r\}$ is a subset of neither $X\setminus\{r\}$ nor $Y$. Hence, there is a facet $F'$ of $\Delta(T)$ containing $F$ and being different from $X\setminus\{r\}$ and $Y$. Therefore,
$$\sum_{i\notin F}\al_i \geqslant \sum_{i\notin F}\bt_i+1 \geqslant \sum_{i\notin F'}\bt_i+1 \geqslant s+1=n.$$
Which implies $F\notin \Delta_{\alb}(I(H)^n)$.

If $t-1\notin F$, then $(F\cup\{t\})\setminus\{r\}$ is a facet of $\Delta(T)$. Noticing that $\al_{t-1}=\bt_{t-1}+1$ and $\al_t=1$, we get
$$\sum_{i\notin F}\al_i = \sum_{i\notin (F\cup\{t\})\setminus\{r\}}\bt_i +1+\al_t\geqslant  (s-1)+2=n.$$
Which again implies $F\notin \Delta_{\alb}(I(H)^n)$.

{\it Case $2$}: $t+1< r$. Thus $\btb = (\al_1,\ldots, \al_{r-1})$, and thus $s = n$. Now we can proceed as in Case $1$. This completes the proof of Claim $1$.

{\it Claim $2$}: $\Delta_{\alb}(I(G)^n) =\left<X,Y\right>.$

{\it Proof:} by Claim $1$ and Lemma $\ref{L02}$, $X$ and $Y$ are facets of $\Delta_{\alb}(I(G)^n)$. It remains to show that for any facet $F$ of $\Delta(G)$ being different from $X$ and $Y$, then $F\notin\Delta_{\alb}(I(G)^n)$. Since $F$ is a face of $H$, we have $F\subseteq F'$ for some facet $F'$ of $\Delta(H)$. Then, $F'$ is different from $X$ and $Y$, and then $F' \notin \Delta_{\alb}(I(H)^n)$. Thus, by Lemma $\ref{L02}$ we have
$$\sum_{i\notin F} \al_i \geqslant \sum_{i\notin F'} \al_i \geqslant n$$
and thus $F \not \in \Delta_{\alb}(I(G)^n)$, as claimed.

Now we return to the proof of the proposition. Claim $2$ and Lemma $\ref{LT1}$ give $\dstab(I(G))\leqslant n$, or equivalently
$$\dstab(I(G)) \leqslant  \varepsilon(H)-\varepsilon_0(H)-k+1.$$
Let $e$ be an edge of the cycle $C_{2k}$. Then, $H\setminus e$ is a tree. Hence, by  \cite[Theorem $4.3$]{BM} we have
$\varepsilon(H) = \varepsilon(H\backslash e) +1 = (\upsilon(H\backslash e)-1)+1 = \upsilon(H\backslash e)=\upsilon(H) = \upsilon(G)$. Clearly, $\varepsilon_0(G) \leqslant \varepsilon_0(H)$. Therefore, $$\dstab(I(G)) \leqslant \varepsilon(H)-\varepsilon_0(H)-k+1 \leqslant \upsilon(G)-\varepsilon_0(G)-k+1,$$ as required.
\end{proof}

\section{Depths of powers of edge ideals}

In this section we study the stability of $\depth R/I(G)^n$ for any graph $G$.  First we need some basic facts of homological modules of simplicial complexes. 

A tool which will be of much use is the Mayer-Vietoris sequence, see \cite[Theorem $25.1$]{MU} or \cite[in Page $21$]{ST} page $21$. For two simplicial complexes $\Delta_1$ and $\Delta_2$, we have the long exact sequence of reduced homology modules
$$ \cdots \rightarrow \h_i(\Delta_1) \oplus \h_i(\Delta_2)  \rightarrow \h_i(\Delta)\rightarrow \h_{i-1}(\Delta_1 \cap \Delta_2)\rightarrow  \h_{i-1}(\Delta_1) \oplus \h_{i-1}(\Delta_2) \rightarrow \cdots
$$
where $\Delta = \Delta_1 \cup \Delta_2$.

A simplicial complex $\Delta$ is a cone if there is a vertex $v$ such that $\{v\} \cup F \in \Delta$ for every $F\in \Delta$. If $\Delta$ is a cone, then it is acylic (see \cite[Theorem $8.2$]{MU}), i.e., 
$$\h_i(\Delta; K) = \zv \text{ for every } i\in \Z.$$

Finally, for two simplicial complexes $\Delta$ and $\Gamma$ over two disjoint vertex sets, the join of $\Delta$ and $\Gamma$, denoted by $\Delta * \Gamma$, is defined by 
$$\Delta * \Gamma := \{F \cup G \mid F\in \Delta \text{ and  } G\in\Gamma\}.$$

\begin{lem}\label{L08} Let $G$ be a bipartite graph with connected components $G_1,\ldots,G_s$ and let $n := \sum_{i=1}^s \dstab(I(G_i))-s+1$. Then there is $\alb = (\alb_1,\ldots,\alb_r) \in \N^r$ such that
$$\sum_{i\notin F}\al_i = n-1 \ \text{ for all } F \in\mathcal F(\Delta_{\alb}(I(G)^n))  \ \text{ and } \widetilde{H}_{s-1}(\Delta_{\alb}(I(G)^n); K) \ne \zv.$$
\end{lem}
\begin{proof} For each $i$, let $(X_i,Y_i)$ be a bipartition of $G_i$ and $n_i := \dstab(I(G_i))$, so that
$$n = \sum_{i=1}^s n_s -s +1.$$

Since the vertex sets of $G_1,\ldots,G_s$ are mutually disjoint, by Lemma $\ref{LT1}$ there is $\alb =(\al_1,\ldots,\al_r)\in \N^r$ such that
\begin{equation}\label{EQ08}
\sum_{j\in V(G_i)\setminus X_i}\al_j = \sum_{j\in V(G_i)\setminus Y_i}\al_j = n_i-1,
\end{equation}
and
\begin{equation}\label{EQ09}
\sum_{j\in V(G_i)\setminus F_i}\al_j \geqslant  n_i \ \text{ for all } F_i \in \mathcal F(\Delta(G_i)) \setminus \{X_i,Y_i\}.
\end{equation}
For any $F \in\mathcal F(\Delta(G))$, we can partition $F$ into $F = \bigcup_{i=1}^s F_i$ where $F_i \in \mathcal F(\Delta(G_i))$ for $i = 1,\ldots,s$. By Equation ($\ref{EQ08}$) and Inequality ($\ref{EQ09}$) we get
$$\sum_{j\notin F}\al_j = \sum_{i=1}^s \sum_{j\in V(G_i) \setminus F_i}\al_j \geqslant \sum_{i=1}^s (n_i-1) = n -1$$
and the equality occurs if and only if
$$\sum_{j\in V(G_i)\setminus F_i}\al_j = n_i -1 \ \text{ for all } i =1,\ldots,s,$$
or equivalently, either $F_i=X_i$ or $F_i=Y_i$ for all $i = 1,\ldots, s$. Together with Lemma $\ref{L02}$ we have
$$\sum_{j\notin F}\al_j = n -1 \ \text{ for all } F \in\mathcal F(\Delta_{\alb}(I(G)^n)),$$
and
$$ \Delta_{\alb}(I(G)^n) = \left<X_1,Y_1\right> * \cdots * \left<X_s,Y_s\right>.$$
So it remains to prove that $\widetilde H_{s-1}(\left<X_1,Y_1\right> * \cdots * \left<X_s,Y_s\right>; K) \ne \zv$. In order to prove this, let $\Delta_i := \left<X_1,Y_1\right> * \cdots * \left<X_i,Y_i\right>$ for $i=1,\ldots,s$ and $\Delta_0 :=\{\emptyset\}$. Then, for all $i=1,\ldots,s$ we have
$$\Delta_{i} = \left<X_i\right> * \Delta_{i-1} \cup \left<Y_i\right> * \Delta_{i-1} \ \text{ and } \Delta_{i-1} = \left<X_i\right> * \Delta_{i-1} \cap \left<Y_i\right> * \Delta_{i-1}.$$

Since $\left<X_i\right> * \Delta_{i-1}$ and $\left<X_i\right> * \Delta_{i-1}$ are cones, by using Mayer-Vietoris sequence, we get an exact sequence $\zv \rightarrow \widetilde H_{s-1}(\Delta_s; K)  \rightarrow \widetilde H_{s-2}(\Delta_{s-1}; K)\rightarrow \zv$. Thus,
$$\widetilde H_{s-1}(\Delta_s; K) \cong \widetilde H_{s-2}(\Delta_{s-1}; K).$$
By repeating this way we obtain
$$\widetilde H_{s-1}(\Delta_s; K) \cong \widetilde H_{s-2}(\Delta_{s-1}; K) \cong \cdots\cong \widetilde H_{-1}(\Delta_0; K) \cong K,$$
and so $\widetilde H_{s-1}(\Delta_s; K) \ne \zv$, as required.
\end{proof}

The next lemma gives the limit of the sequence $\depth R/I(G)^n$.

\begin{lem} \label{LT2} Let $G$ be a graph. Assume that $G_1,\ldots,G_s$ are all connected bipartite components of $G$ and $G_{s+1},\ldots,G_{s+t}$ are all connected nonbipartite components of $G$. Then
$$\depth R/I(G)^n = s \ \text{ for all } n\geqslant \sum_{i=1}^{s+t}\dstab(I(G_i)) -(s+t)+1.$$
\end{lem}
\begin{proof} Let $n_i := \dstab(I(G_i))$ for $i=1,\ldots,s+t$. We divide the proof into three cases:

{\it Case} $1$. $s=0$, i.e., every component of $G$ is nonbipartite. This case follows from Lemma $\ref{LT21}$.

{\it Case} $2$. $t=0$, i.e., $G$ is bipartite. Let $m:= \sum_{i=1}^s n_i - s + 1$. By Lemmas $\ref{L01}$ and $\ref{L08}$, there is $\alb\in \N^r$ such that
$$\dim_K H_{\mi}^s(R/I(G)^m)_{\alb} = \dim_K \widetilde{H}_{s-1}(\Delta_{\alb}(I(G)^m); K) \ne 0.$$
Hence, $H_{\mi}^s(R/I(G)^m)\ne \zv$, which yields $\depth R/I(G)^m\leqslant s$. On the other hand, by Lemma $\ref{C1}$ we have $\depth R/I(G)^m\geqslant s$. Thus, $\depth R/I(G)^m= s$. The lemma now follows from Lemma $\ref{C1}$.

{\it Case} $3$. $s\ne 0$ and $t \ne 0$. Let $G'$ and $G''$ be induced subgraphs of $G$ defined by
$$G' :=\bigcup_{i=1}^s G_i \ \text{ and } G'' := \bigcup_{i=1}^t G_{s+i}.$$
We may assume that $V(G') = [p]$ and $V(G'') = \{p+1,\ldots,p+q\}$, where $p+q = r$.  For simplicity, we set $y_1:=x_{p+1},\ldots,y_q:=x_{p+q}$. Then $R = K[x_1,\ldots, x_p,y_1,\ldots,y_q]$. Let $R' := K[x_1,\ldots, x_p], R'' := K[y_1,\ldots,y_q]$,  $m := \sum_{i=1}^s n_i -s +1$ and $n_0 := n-m+1$. Note that $n_0 \geqslant \sum_{i=1}^t n_{s+i}- t + 1$, so $(y_1,\ldots, y_q) \in \ass(R''/I(G'')^{n_0} )$  by Lemma $\ref{LT21}$. Accordingly, there exists $\btb = (\beta_1,\ldots,\beta_q)\in \N^q$ such that $(y_1,\ldots, y_q) = I(G'')^{n_0} : \y^{\btb}$. This implies
\begin{equation}\label{EQ17} \y^{\btb} \in I(G'')^{n_0-1},\  \y^{\btb} \notin I(G'')^{n_0} \ \text{ and } \y^{\btb} \in I(G'')^{n_0}_F \text{ whenever } \emptyset \ne F\in\Delta(G'').\end{equation}
Next, by Lemma $\ref{L08}$ there is $\alb = (\al_1,\ldots \al_p) \in \N^p$ such that
\begin{equation}\label{EQ18}
\widetilde{H}_{s-1}(\Delta_{\alb}(I(G')^m); K) \ne \zv, \ \text{ and }\sum_{i\notin V}\al_i = m-1 \ \text{ for all } V \in\mathcal F(\Delta_{\alb}(I(G')^m)).
\end{equation}
Let $\gmb := (\al_1,\ldots,\al_p, \beta_1,\ldots,\beta_q) \in \N^r$. Note that $\x^{\gmb} = \x^{\alb} \y^{\btb} \in R$. We claim that
\begin{equation}\label{EQ19}
\Delta_{\gmb}(I(G)^n) = \Delta_{\alb}(I(G')^m).
\end{equation}
Indeed, for all $H \in \Delta_{\gmb}(I(G)^n)$ we can partition $H$ into $H = H_1 \cup H_2$ where
$H_1 \in \Delta(G')$ and $H_2 \in \Delta(G'')$. By Equation ($\ref{EQ01}$) we have\begin{equation}\label{EQ20} \x^{\gmb} = \x^{\alb}\y^{\btb} \notin I(G)^n_H = (I(G')_{H_1}+I(G'')_{H_2})^n=\sum_{i=0}^n I(G')_{H_1}^i I(G'')_{H_2}^{n-i}.\end{equation}
Now, if $H_2\ne\emptyset$, then by Formula ($\ref{EQ17}$) we would have $\y^{\btb}\in I(G'')^{n_0}_{H_2}$. Then, Formula ($\ref{EQ20}$) forces $\x^{\alb} \notin I(G')_{H_1}^{n-n_0}=I(G')_{H_1}^{m-1}$, thus $H_1 \in\Delta_{\alb}(I(G')^{m-1})$. In particular, $\Delta_{\alb}(I(G')^{m-1}) \ne \emptyset$. Let us take arbitrary facet $V$ of $\Delta_{\alb}(I(G')^{m-1})$. By Lemma $\ref{L02}$ we then have $\sum_{i\notin V}\al_i \leqslant m-2$. By Lemma $\ref{L02}$ again, $V$ is a facet of $\Delta_{\alb}(I(G')^{m})$, which contradicts $(\ref{EQ18})$. Thus, $H_2 =\emptyset$ and $H = H_1$. Formula ($\ref{EQ20}$) now becomes
$$\x^{\gmb} = \x^{\alb}\y^{\btb} \notin (I(G')_{H}+I(G''))^n = \sum_{i=0}^n I(G')^i_H I(G'')^{n-i}.$$
Together with Formula ($\ref{EQ17}$), this fact implies $\x^{\alb} \notin I(G')^{n-n_0+1}_H = I(G')^{m}_H$, or equivalently, $H\in \Delta_{\alb}(I(G')^m)$, so $\Delta_{\gmb}(I(G)^n) \subseteq \Delta_{\alb}(I(G')^m)$.

In order to prove the reverse inclusion, suppose that $H \in \Delta_{\alb}(I(G')^m)$. Then, $\x^{\alb} \notin I(G')^m_H$ by Equation ($\ref{EQ01}$). If $\x^{\gmb}\in I(G)^n_H$, then
$$\x^{\gmb} = \x^{\alb}\y^{\btb} \in I(G)^n_H = (I(G')_H+I(G''))^n =\sum_{i=0}^n I(G')_H^i I(G'')^{n-i}.$$
Hence, $\x^{\alb}\y^{\btb} \in I(G')^{\nu}_H I(G'')^{n-\nu}$ for some nonnegative integer $\nu$. Since $V(G') \cap V(G'')=\emptyset$, it yields $\x^{\alb} \in I(G')^{\nu}_H$ and $\y^{\btb}\in I(G'')^{n-\nu}$. By Formula ($\ref{EQ17}$) we deduce that $n - \nu \leqslant n_0-1$, and so $\nu \geqslant n - n_0 + 1 = m$. But then $\x^{\alb}\in I(G')^m_H$, a contradiction. Hence, $\x^{\gmb}\notin  I(G)^n_H$, i.e., $H\in\Delta_{\gmb}(I(G)^n)$, and hence $\Delta_{\alb}(I(G')^m) \subseteq \Delta_{\gmb}(I(G)^n)$, as claimed.

Combining Formulas ($\ref{EQ18}$) and ($\ref{EQ19}$) with Lemma $\ref{L01}$, we get
$$\dim_K H_{\mi}^s(R/I(G)^n)_{\gmb} = \dim_K \widetilde H_{s-1}(\Delta_{\gmb}(I(G)^n); K) = \dim_K \widetilde H_{s-1}(\Delta_{\alb}(I(G')^m); K)\ne 0.$$
Therefore, $H_{\mi}^s(R/I(G)^n)\ne\zv$, so
\begin{equation}\label{EQ21}\depth R/I(G)^n \leqslant s.\end{equation}
On the other hand, since $G'$ is bipartite, by Lemmas $\ref{C1}$ and $\ref{L10}$ we get
$$\depth R/I(G)^n =\depth R/(I(G')+I(G''))^n \geqslant \min_{\nu\geqslant 1} \depth R'/I(G')^{\nu} = s.$$
Together with Inequality ($\ref{EQ21}$), we obtain $\depth R/I(G)^n = s$, as required.
\end{proof}

\begin{cor} For all graphs $G$ we have $\lim_{n\rightarrow \infty} \depth R/I(G)^n = \dim R -\ell(I(G)).$
\end{cor}
\begin{proof} Let $s$ be the number of bipartite components of $G$. Then $s = \dim R -\ell(I(G))$ (see \cite[Page $50$]{W}), so the corollary immediately follows from Lemma $\ref{LT2}$.
\end{proof}

We are now ready to prove the first main result of the paper.

\begin{thm} \label{MP} Let $G$ be a graph with $p$ connected components $G_1,\ldots, G_p$. Let $s$ be the number of connected bipartite components of $G$. Then
\begin{enumerate}
\item $\min\{\depth R/I(G)^n\mid n\geqslant 1\} = s$.
\smallskip
\item $\dstab(I(G)) = \min\{n\geqslant 1 \mid \depth R/I(G)^n = s\}.$
\smallskip
\item $\dstab(I(G)) = \sum_{i=1}^p \dstab(I(G_i))-p+1$.
\end{enumerate}
\end{thm}
\begin{proof} We may assume that  $G_1,\ldots,G_s$ are bipartite.

$(1)$ If $s = 0$ (resp. $s=p$), then the first statement follows from Lemma $\ref{LT21}$ (resp. Lemma $\ref{C1}$). Assume that $1\leqslant s < p$. Let $G'$ be the induced subgraph of $G$ consisting of $G_1,\ldots,G_s$ and $G''$ the induced subgraph of $G$ consisting of $G_{s+1},\ldots,G_p$. Then, $I(G) = I(G')+I(G'')$. Let $R' :=K[x_i\mid i\in V(G')]$. For all $n\geqslant 1$, since $G'$ is bipartite,  by Lemmas $\ref{C1}$ and $\ref{L10}$ we have
$$\depth R/I(G)^n\geqslant \min\{\depth R'/I(G')^m \mid m\geqslant 1\} =s.$$
Together with Lemma $\ref{LT2}$ we conclude that
$$\min\{\depth R/I(G)^n \mid n\geqslant 1\} = s,$$
and $(1)$ follows.

We next prove ($2$) and ($3$) simultaneously by induction on $p$. If $p=1$, then the theorem follows from Lemmas $\ref{LT21}$ and $\ref{C1}$.

Assume that $p\geqslant 2$. If $s = 0$, our claim follows from Lemma $\ref{LT21}$. So we may assume that $s \geqslant 1$. Let $H$ be the induced subgraph of $G$ consisting of components $G_2,\ldots, G_p$. Then, $H$ has $p-1$ connected components and $s-1$ connected bipartite components.   By Lemma $\ref{LT2}$ we have
$$ \depth R/I(G)^n = s \text{ for all } n\geqslant \sum_{i=1}^p \dstab(I(G_i))-p+1.$$
Hence, in order to prove the theorem  it suffices to show that if
\begin{equation}\label{CONDITION}
\depth R/I(G)^n = s
\end{equation}
for a given positive integer $n$,  then $n\geqslant \sum_{i=1}^p \dstab(I(G_i))-p+1$.

In order to prove this assertion let $A:=K[x_j\mid j\in V(G_1)]$ and $B:=K[x_j\mid j\in V(H)]$. Then, we have $\dim A \geqslant 2$ and $\dim B\geqslant s$. For simplicity, we set $I := I(G_1)$ and $J := I(H)$. We now claim that
\begin{equation}\label{X1} \depth R/I^iJ^{n-i} \geqslant  s+1 \text{ for } i=0,\ldots,n.
\end{equation}
Indeed, if $i =n$, since $\depth A/I^n\geqslant 1$ and $\dim B\geqslant s$, we have
$$\depth R/I^nJ^0 = \depth R/I^n = \depth A/I^n +\dim B \geqslant 1 +s.$$
Since $\depth B/J^n \geqslant s-1$ by Part $1$, a similar proof also holds for $i=0$. For all $i=1,\ldots,n-1$, by \cite[Lemma $2.2$]{HT1} we have $\depth R/I^iJ^{n-i} = \depth A/I^i + \depth B/J^{n-i} +1$. Hence, $\depth R/I^iJ^{n-i} \geqslant 1+(s-1)+1 = s+1$, as claimed.

Let $n_1:=\dstab(G_1)$ and $n_2:=\dstab(H)$. We will prove that $n\geqslant n_1+n_2-1$. Assume on the contrary that $n \leqslant n_1+n_2-2$. For each $i=0,\ldots,n$, we put
$$W_i :=I^iJ^{n-i}+\cdots+I^nJ^0,$$
where $I^0 = J^0 = R$. We next claim that
\begin{equation}\label{LM}
\depth R/W_i \geqslant s+1 \text{ for all } i=0,\ldots,n.
\end{equation}
Indeed, we prove this by induction on $i$. If $i = n$, then by Inequality ($\ref{X1}$) we have
$$\depth R/W_n = \depth R/I^n \geqslant s+1.$$
Assume that $\depth R/W_{i+1} \geqslant s+1$ for some $0\leqslant i < n$. By Equations ($\ref{M1}$) and ($\ref{M2}$), we have $I^iJ^{n-i}\cap W_{i+1} = I^{i+1}J^{n-i}$. Since $W_i = I^iJ^{n-i} +W_{i+1}$, we have  an exact sequence
$$\zv \longrightarrow R/I^{i+1}J^{n-i} \longrightarrow R/I^iJ^{n-i} \oplus R/W_{i+1} \longrightarrow R/W_i \longrightarrow \zv.$$
By Depth Lemma, we have
$$\depth R/W_i \geqslant \min\{\depth R/I^{i+1}J^{n-i}-1,  \depth R/I^iJ^{n-i}, \depth R/W_{i+1}\}.$$
Together with Inequality ($\ref{X1}$) and the induction hypothesis, this fact yields
$$\depth R/W_i \geqslant \min\{\depth R/I^{i+1}J^{n-i}-1, s+1\}.$$
Therefore, the inequality ($\ref{LM}$) will follows if $\depth R/I^{i+1}J^{n-i} \geqslant s+2$. In order to prove this inequality, note that $(i+1)+(n-i) = n+1 \leqslant n_1 + n_2 -1$. Hence, either $i+1 < n_1$ or $n-i < n_2$. Note that $n-i\geqslant 1$.

If $i+1 < n_1$, by Part $1$ we get $\depth A/I^{i+1} \geqslant 2$ and $\depth B/J^{n-i}\geqslant s-1$. Together with \cite[Lemma $2.2$]{HT1} we obtain
$$\depth R/I^{i+1}J^{n-i} = \depth A/I^{i+1} + \depth B/J^{n-i}+1 \geqslant 2 + (s-1)+1=s+2,$$
as claimed.

If $n-i < n_2$, the proof is similar. Thus, the claim ($\ref{LM}$) is proved.

Notice that $W_0 = (I+J)^n = (I(G_1)+I(H))^n = I(G)^n$. By ($\ref{LM}$) we have $\depth R/I(G)^n \geqslant s+1$. This contradicts $(\ref{CONDITION})$. Therefore, we must have  $n\geqslant n_1+n_2-1$.

Finally, by the induction hypothesis we have
$$n_2 = \dstab(I(H)) = \sum_{i=2}^{p}\dstab(I(G_i)) -(p-1)+1.$$
Together with $n_1=\dstab(I(G_1))$, we have
$$n\geqslant n_1+n_2-1 =  \sum_{i=1}^p\dstab(G_i) -p+1,$$
as required.
\end{proof}

\begin{rem} From Theorem $\ref{MP}$ and Lemmas $\ref{LT21}$ and $\ref{LT1}$ we see that $\dstab(I(G))$ is independent from the characteristic of the base field $K$, so it depends purely on the structure of $G$.
\end{rem}

We next combine Theorem $\ref{MP}$ and Propositions $\ref{A3}$ and $\ref{L07}$ to get the second main result of the paper, which sets up an upper bound for $\dstab(I(G))$.

\begin{thm} \label{MT} \it Let $G$ be a graph. Let $G_1,\ldots,G_s$ be all connected bipartite components of $G$ and let $G_{s+1},\ldots,G_{s+t}$ be all connected nonbipartite components of $G$. Let $2k_i$ be the maximum length of cycles of $G_i$ ($k_i :=1$ if $G_i$ is a tree) for all $i=1,\ldots,s$; and let $2k_i-1$ be the maximum length of odd cycles of $G_i$ for every $i = s+1,\ldots, s+t$. Then
$$\dstab(I(G)) \leqslant \upsilon(G) - \varepsilon_0(G) -\sum_{i=1}^{s+t} k_i +1.$$
\end{thm}
\begin{proof} Since
$$\upsilon(G) - \varepsilon_0(G) -\sum_{i=1}^{s+t} k_i +1 = \sum_{i=1}^{s+t} \left(\upsilon(G_i)-\varepsilon_0(G_i)-k_i+1\right) -(s+t) +1,$$
by Propositions $\ref{A3}$ and $\ref{L07}$ we get
$$\upsilon(G) - \varepsilon_0(G) -\sum_{i=1}^{s+t} k_i +1 \geqslant \sum_{i=1}^{s+t} \dstab(I(G_i)) -(s+t) +1.$$
Together with Theorem $\ref{MP}$ we obtain
$$\dstab(I(G)) =  \sum_{i=1}^{s+t}\dstab(I(G_i)) - (s+t) +1 \leqslant \upsilon(G) - \varepsilon_0(G) -\sum_{i=1}^{s+t} k_i +1,$$
as required.
\end{proof}

\section{The index of depth stability of trees and unicyclic graphs}

The aim of this section is to prove that the upper bound of $\dstab(I(G))$ given in Theorem $\ref{MT}$ is always achieved if $G$ has no cycles of length $4$ and every component of $G$ is either a tree or a unicyclic graph. Recall that a connected graph $G$ is a tree if it contains no cycles; and $G$ is a unicyclic graph if it contains exactly one cycle. 

If $G$ is a unicyclic graph and $C$ is the unique cycle of $G$, then for every vertex $v$ of $G$ not lying in $C$, there is a unique simple path of minimal distance from $v$ to a vertex in $C$.

\begin{thm} \label{DMP} Let $G$ be a graph with $p$ connected components $G_1,\ldots,G_p$ such that each $G_i$ is either a tree or a unicyclic graph. For each $i$, if $G_i$ is bipartite, let $2k_i$ be the length of its unique cycle ($k_i :=1$ if $G_i$ is a tree); and if $G_i$ is nonbipartite, let $2k_i-1$ be the length of its unique cycle. If $G$ has no cycles of length $4$, then
$$\dstab(I(G)) =  \upsilon(G) - \varepsilon_0(G) -\sum_{i=1}^p k_i +1.$$
\end{thm}
By Theorem $\ref{MT}$, it suffices to show that $\dstab(G_i) = \upsilon(G_i)-\varepsilon_0(G_i)-k_i+1$ for each $i=1,\ldots,p$. If $G_i$ is nonbipartite, the equality follows from Lemma $\ref{T6}$. Thus,  it remains to prove this equality for the case $G_i$ is bipartite.

We divide the proof into two lemmas. The first lemma deals with unicyclic bipartite graphs and the second one deals with trees.

For a vertex $x$ of $G$, we denote $L_G(x)$ to be the set of leaves of $G$ that are adjacent to $x$. We start with the following observation.

\begin{lem}\label{LT3} Let $G$ be a graph with $r = \upsilon(G)$. Let $p$ be a leaf of $G$ and $q$ the unique neighbor of $p$ in $G$. Let $\alb = (\al_1,\ldots,\al_r) \in \N^r$ and we define $\btb=(\bt_1,\ldots,\bt_r)$ by
$$ \bt_i:= \left\{ \begin{array}{ll}
\al_i  +1 & \mbox{if $i=p$ or $i=q$},\\
\al_i       & \mbox{otherwise}.
\end{array}
\right.
$$
Then $\Delta_{\alb}(I(G)^n) = \Delta_{\btb}(I(G)^{n+1})$ for all $n \geqslant 1$.
\end{lem}
\begin{proof} Let $F$ be a facet of $\Delta(G)$. By the maximality of $F$, it must contain either $p$ or $q$ but not both, so
$$\sum_{i\notin F}\bt_i = \sum_{i\notin F}\al_i +1.$$
Thus, by Lemma $\ref{L02}$ we get $\Delta_{\alb}(I(G)^{n}) = \Delta_{\btb}(I(G)^{n+1})$ for all $n \geqslant 1$.
\end{proof}

\begin{lem}\label{T4} Let $G$ be a unicyclic bipartite graph. Assume that the unique cycle of $G$ is $C_{2k}$ of length $2k$ with $k\geqslant 3$. Then, $\dstab(I(G)) = \upsilon(G)-\varepsilon_0(G)-k+1$.
\end{lem}
\begin{proof} Let $n:=\dstab(I(G))$. By Theorem $\ref{MT}$ we have $n \leqslant \upsilon(G)-\varepsilon_0(G)-k+1$. Thus, in order to prove the theorem it suffices to show $n \geqslant \upsilon(G)-\varepsilon_0(G)-k+1$.

Let $(X, Y)$ be a bipartition of $G$. Then, by Lemma $\ref{LT1}$ there is $\alb =(\al_1,\ldots,\al_r)\in \N^r$ such that
\begin{equation} \label{UN1} \Delta_{\alb}(I(G)^n) = \left<X,Y\right>  \text{ and } \sum_{j\in X} \al_j= \sum_{j\in Y} \al_j= n-1.
\end{equation}
Observe that for any face $F$ of $\Delta(G)$ with $F \cap X \ne \emptyset$ and $F \cap Y \ne \emptyset$, we have
\begin{equation}\label{UN2}   \sum_{i\notin F} \al_i \geqslant n.\end{equation}
Indeed,  let $L$ be a facet of $\Delta(G)$ which contains $F$, so that $L$ meets both $X$ and $Y$. Since $\Delta_{\alb}(I(G)^n) = \left<X,Y\right>$, $L \notin \Delta_{\alb}(I(G)^n)$.  By Lemma $\ref{L02}$ we get
$$\sum_{i\notin F}\al_i\geqslant \sum_{i\notin L} \al_i \geqslant n,$$
and the formula $(\ref{UN2})$ follows.

We now prove $n \geqslant \upsilon(G)-\varepsilon_0(G)-k+1$ by induction on $\upsilon(G)$.

If $\upsilon(G)=2k$, i.e.,  $G=C_{2k}$, then $\upsilon(G)-\varepsilon_0(G)-k+1 = k+1$. For each $i \in X$, let $N_G(i) = \{u_i,v_i\}$ and $F_i := \{i\} \cup (Y\setminus \{u_i,v_i\})$. Then, $F_i\in\Delta(G)$. Since $|X| = |Y| = k \geqslant 3$, $F_i\cap X\ne\emptyset$ and $F_i\cap Y\ne \emptyset$. Together with Formulas $(\ref{UN1})$ and $(\ref{UN2})$, this fact gives
$$n\leqslant \sum_{j\notin F_i}\al_j = \sum_{j\in X}\al_j + \al_{u_i}+\al_{v_i}-\al_i = n-1+\al_{u_i}+\al_{v_i}-\al_i,$$
whence $\al_i+1\leqslant \al_{u_i}+\al_{v_i}$. Hence,
$$\sum_{i\in X}\al_i+k = \sum_{i\in X}(\al_i+1) \leqslant \sum_{i\in X}(\al_{u_i}+\al_{v_i})  = 2\sum_{j\in Y}\al_j.$$
Together with Formula ($\ref{UN1}$), this gives $(n-1)+k \leqslant 2(n-1)$. Thus, $n \geqslant k+1$, and thus the lemma holds for this case.

Assume that $\upsilon(G) > 2k$. We distinguish two cases:

\smallskip

{\it Case $1$:} $G\setminus V(C_{2k})$ is totally disconnected. For any vertex $u$ lying in $C_{2k}$ with $L_G(u)\ne \emptyset$, we claim that 
\begin{equation}\label{UN3Wisker} \al_u \geqslant 1 \text{, and } \al_i = 0 \ \text{ for every } i \in L_G(u).\end{equation}

Indeed, without loss of generality we may assume that $u\in Y$, so that $L_G(u)\subseteq X$. let $F := (Y \setminus \{u\}) \cup L_G(u)$. Then, $F \in \Delta(G)$. Since the length of $C_{2k}$ is at least $6$, we have $F\cap Y\ne \emptyset$. Notice that $\emptyset \ne L_G(u) \subseteq F \cap X$.  Therefore, $F\cap X \ne \emptyset$ and $F\cap Y\ne \emptyset$. By Formula $(\ref{UN2})$ we have
$$\sum_{i\in X}\al_i +\al_u - \sum_{i\in L_G(u)} \al_i = \sum_{i\notin F}\al_i \geqslant n.$$
By ($\ref{UN1}$), this gives
$$n-1 +\al_u - \sum_{i\in L_G(u)} \al_i  \geqslant n,$$
so
$$\al_u \geqslant  \sum_{i\in L_G(u)}\al_i +1\geqslant 1.$$
Hence, it remains to prove that $\al_i = 0$ for all $i\in L_G(u)$. Assume that $\al_i \geqslant 1$ for some $i\in L_G(u)$. Define $\btb = (\bt_1,\ldots,\bt_r)$ by
$$ \bt_j := \left\{ \begin{array}{ll}
\al_j  - 1 & \mbox{if $j=u$ or $j=i$},\\
\al_j       & \mbox{otherwise}.
\end{array}
\right.
$$
Then, $\btb\in \N^r$. Since $u\in Y$ and $\alpha_u\geqslant 1$, by ($\ref{UN1}$) we have
$$n-1 =\sum_{j\in Y}\al_j \geqslant \al_u \geqslant 1.$$
By Lemma $\ref{LT3}$ we have $\Delta_{\btb}(I(G)^{n-1}) = \Delta_{\alb}(I(G)^n)$. Consequently, $\Delta_{\btb}(I(G)^{n-1}) =\left<X,Y\right>$, which implies $\depth R/I(G)^{n-1} = 1$ by Lemma $\ref{LT1}$, and so $\dstab(I(G)) \leqslant n-1$ by Theorem $\ref{MP}$. This contradicts to $n=\dstab(I(G))$. Thus, $\al_i = 0$, as claimed.

We may assume that $V(H) = \{1,\ldots,2k\}$. Let $\btb := (\alpha_1,\ldots,\alpha_{2k}) \in \N^{2k}$, $X_0 := X \cap V(C_{2k})$ and $Y_0 := Y\cap V(C_{2k})$. Then, $(X_0,Y_0)$ is a bipartition of $C_{2k}$. Clearly,
$$
X = X_0 \cup \bigcup_{i\in Y_0} L_G(i) \text{ and } Y = Y_0 \bigcup _{i\in X_0} L_G(i).
$$
Together with Claim $(\ref{UN3Wisker})$ we have
$$\sum_{i\notin  X_0} \beta_i = \sum_{i\notin X}\alpha_i = n-1.$$
Similarly, $\sum_{i\notin  Y_0} \beta_i = n-1$. Therefore, $X_0, Y_0 \in \Delta_{\btb}(I(C_{2k})^n)$.

For any facet $F$ of $\Delta(C_{2k})$ which is different from $X_0$ and $Y_0$, let $$F' := F \cup \bigcup_{i\in V(C)\setminus F} L_G(i).$$
Then, $F'$ is a facet of $\Delta(G)$ which is different from $X$ and $Y$. Together Claim $(\ref{UN3Wisker})$ with Lemma $\ref{L02}$, we have
$$\sum_{i\notin F} \beta_i = \sum_{i\notin F'} \alpha_i  \geqslant n$$
so that $F\notin \Delta_{\btb}(I(C_{2k})^n)$. Thus, $\Delta_{\btb}(I(C_{2k})^n) =\left<X_0, Y_0\right>$.

This gives $\depth S/I(C_{2k})^n = 1$ where $S=K[x_1,\ldots,x_{2k}]$. From the case $\upsilon(G) = 2k$ above,  we imply that
$$n \geqslant k+1 = \upsilon(G) -\varepsilon_0(G)-k+1,$$
and the lemma holds in this case.

\medskip

{\it Case $2$:} $G\setminus V(C_{2k})$ is not totally disconnected. Let $v$ be a leaf of $G$ such that $d_G(v,C_{2k})$ is maximal. By Remark $\ref{rem_path}$, we deduce that $N_G(v)$ has only one non-leaf, say $u$, and $N_G(u)$ also has only one non-leaf, say $w$. Note that $L_G(u) \ne\emptyset$ since $v\in L_G(u)$. We may assume that $u\in Y$, so that $v\in X$. We first claim that
\begin{equation}\label{UN3} \al_u \geqslant 1 \text{, and } \al_i = 0 \ \text{ for every } i \in L_G(u).\end{equation}
Indeed, let $F := (Y \setminus \{u\}) \cup L_G(u)$. Then, $F \in \Delta(G)$. Since $|N_G(w)| \geqslant 2$ and $N_G(w)\subseteq Y$, we have $\emptyset \ne N_G(w)\setminus\{u\}\subseteq Y\setminus \{u\} \subseteq F\cap Y$. Notice that $\emptyset \ne L_G(u) \subseteq F \cap X$.  Therefore, $F\cap X \ne \emptyset$ and $F\cap Y\ne \emptyset$. The proof of claim now carries out the same as in Claim $(\ref{UN3Wisker})$.

We next claim that
\begin{equation}\label{UN4} \al_w \geqslant 1.\end{equation}
Indeed, assume on the contrary that $\al_w = 0$. Note that $w\in X$ and $N_G(u) = L_G(u) \cup\{w\}$. Let $F := (X \cup\{u\}) \setminus N_G(u)$. Then, $F \in \Delta(G)$ and $u\in F \cap Y$. Since $N_G(u) \ne X$, $F \cap X \ne \emptyset$.  By Formulas $(\ref{UN1})-(\ref{UN3})$ and the assumption $\al_w = 0$, these facts give
$$n\leqslant \sum_{i\notin F}\al_i  = \sum_{i\in Y}\al_i  -\al_u +\al_w +\sum_{i\in L_G(u)}\al_i = n-1  -\al_u,$$
and so $\al_u <  0$,  a contradiction. Thus, $\al_w \geqslant 1$, as claimed.

Let $H := G \setminus L_G(u)$. Clearly, $H$ is a connected bipartite graph with bipartition $(X \setminus L_G(u), Y)$. Moreover, $H$ has only cycle $C_{2k}$ as well. We may assume that $V(H) = \{1,\ldots,s\}$. Then $s \geqslant 2k$ and
$L_G(u) = \{s+1,\ldots, r\}$. Let $\ttb = (\theta_1,\ldots,\theta_s) := (\al_1,\ldots, \al_s) \in \N^s$. We now prove
that
\begin{equation}\label{UN5}\Delta_{\ttb}(I(H)^n) = \left<X \setminus L_G(u), Y \right>. \end{equation}
Indeed, by $(\ref{UN3})$ we get  $\sum_{\al_i\in L_G(u)}\al_i = 0$. Together with Formula $(\ref{UN1})$, this fact gives
$$\sum_{i\in V(H), i\notin Y}\theta_i =\sum_{i\in V(H), i\notin Y}\al_i=\sum_{i\in X\setminus L_G(u)} \al_i +\sum_{i\in L_G(u)}\al_i =\sum_{i\in X}\al_i = n-1.$$
Hence, by Lemma $\ref{L02}$, $Y \in\Delta_{\ttb}(I(H)^n)$ . Similarly, $X \setminus L_G(u) \in \Delta_{\ttb}(I(H)^n)$. Now let $F'$ be any facet of $\Delta(H)$ which is different from $X \setminus L_G(u)$ and $Y$ .

If $u \in F'$ then $F'$ is also a facet of $\Delta(G)$. By noticing that $F'$ is different from $X$ and $Y$ and $\sum_{i\in L_G(u)}\al_i=0$, so by $(\ref{UN2})$ we have
$$\sum_{i\in V(H), i\notin F'}\theta_i =\sum_{i\in V(H), i\notin F'}\al_i +\sum_{i\in L_G(u)}\al_i = \sum_{i\notin F'}\al_i\geqslant n,$$
and so $F' \notin\Delta_{\ttb}(I(H)^n)$.

If $u \notin F'$, then $w \in F'$ since $u$ is a leaf of $H$, hence $F' \cup L_G(u)$ is a facet of $\Delta(G)$. Similarly, we have $F' \notin\Delta_{\ttb}(I(H)^n)$, and the formula $(\ref{UN5})$ follows.

\smallskip
Define $\gmb = (\gamma_1,\ldots,\gamma_s)\in\Z^s$ by
$$ \gamma_j := \left\{ \begin{array}{ll}
\theta_j  - 1 & \mbox{if $j=u$ or $j=w$},\\
\theta_j       & \mbox{otherwise}.
\end{array}
\right.
$$
From Inequalities $(\ref{UN3})$ and $(\ref{UN4})$, we have $\gamma_u=\theta_u-1=\al_u-1\geqslant 0$ and $\gamma_w=\theta_w-1=\al_w-1\geqslant 0$, so $\gmb\in\N^s$. Note that
$$n-1= \sum_{i\in X}\al_i \geqslant \al_u \geqslant 1.$$
Therefore, by Lemma $\ref{LT3}$ we have $\Delta_{\gmb}(I(H)^{n-1}) = \Delta_{\ttb}(I(H)^n)$. Together with $(\ref{UN5})$ we get $$\Delta_{\gmb}(I(H)^{n-1}) = \left<X \setminus L_G(1), Y \right>.$$
Hence, by Lemma $\ref{LT1}$ we have $\depth S/I(H)^{n-1} = 1$, where $S=K[x_1,\ldots,x_s]$. By Theorem $\ref{MP}$ we have $\dstab(I(H)) \leqslant n-1$. On the other hand, since $\upsilon(H) =\upsilon(G)-|L_G(u)| <\upsilon(G)$, by the induction hypothesis we have $\dstab(H)\leqslant \upsilon(H)-\varepsilon_0(H)-k+1$. As $\{w,u\}$ is not a leaf edge of $G$ and recall that $H = G \setminus L_G(u)$, we conclude that $\varepsilon_0(G) = \varepsilon_0(H) + |L_G(u)| -1$. Thus,
$$\upsilon(G) -\varepsilon_0(G) - k+1= \upsilon(H) + |L_G(u)|- (\varepsilon_0(H) + |L_G(u)| - 1)-k+1 = \upsilon(H)-\varepsilon_0(H) -k.$$
Hence, $n-1 \geqslant \dstab(I(H)) \geqslant \upsilon(G)-\varepsilon_0(G) -k$, and hence $n \geqslant \upsilon(G)-\varepsilon_0(G) -k+1$.  Thus, the proof now is complete.
\end{proof}

Finally, we compute $\dstab(I(G))$ for trees $G$. If a tree $G$ has a vertex $x$ being adjacent to every other vertex, then $G$ is called a star with a center $x$. Note that $G$ is a star if and only if $\diam(G) \leqslant 2$ where $\diam(G)$ stands for the diameter of $G$. If $\diam(G) = d$, then there is a path $x_1x_2\ldots x_dx_{d+1}$ of length $d$ in $G$. Such a path will be referred to as a path {\it realizing the diameter} of $G$.

\begin{lem}\label{T5} $\dstab(I(G)) = \upsilon(G) -\varepsilon_0(G)$ for all trees $G$.
\end{lem}
\begin{proof} Let $n := \dstab(I(G))$. By Theorem $\ref{MT}$ we have $n \leqslant \upsilon(G) -\varepsilon_0(G)$. So it remains to show $n \geqslant \upsilon(G) -\varepsilon_0(G)$.

If $G$ is a star, then $\varepsilon_0(G) = \varepsilon(G) = \upsilon(G)-1$, and then $\upsilon(G) -\varepsilon_0(G) = 1 \leqslant n$. Thus, the lemma holds for this case.

We will prove  by induction on $\upsilon(G)=r$. If $\upsilon(G) = 2$, then $G$ is one edge, and then it is a star. This case is already proved.

If $\upsilon(G) \geqslant 3$. We may assume that $G$ is not a star so that $\diam G \geqslant 3$. Since $\depth R/I(G)^n = 1$, there is $\alb =(\alpha_1,\ldots,\alpha_r)\in\N^r$ such that $\Delta_{\alb}(I(G)^n) = \left<X,Y\right>$ where $(X,Y)$ is a bipartition of $G$.

Let $vuw \ldots z$ be a path realizing the diameter of $G$. Then $v$ is a leaf, $u$ and $w$ both are not leaves. By \cite[Lemma $3.3$]{MO} we have $N_G(u) = \{w\} \cup L_G(u)$. And now we prove $n\geqslant \upsilon(G)-\varepsilon_0(G)$ by the same way as in Case $2$ in the proof of Lemma $\ref{T4}$. Thus we only sketch the proof here:

First, we show that $\alpha_u\geqslant 1, \alpha_w\geqslant 1$ and $\alpha_i = 0$ for every $i\in L_G(u)$. Then, let $T:=G\setminus L_G(u)$. Note that $T$ is also a tree and  $\upsilon(G)-\varepsilon_0(G) = \upsilon(T)-\varepsilon_0(T)+1$. We may assume that $u\in Y$, $w=s-1$, $u = s$ and $L_G(u) = \{s+1,\ldots,r\}$. Let $\ttb := (\alpha_1,\ldots,\alpha_{s-2}, \alpha_{s-1}-1,\alpha_s-1)\in\N^s$. 
Then, we show that
$$\Delta_{\ttb}(I(T)^{n-1}) =\left<X\setminus L_G(u), Y\right>.$$

This gives $\depth S/I(T)^{n-1} = 1$ where $S=K[x_1,\ldots,x_s]$. By the  induction hypothesis we have
$n-1\geqslant \upsilon(T)-\varepsilon_0(T)$. From that we obtain $n\geqslant  \upsilon(G)-\varepsilon_0(G)$.
\end{proof}

\begin{rem} \label{RMP} Let $G$ be a unicyclic bipartite graph. If the unique circle of $G$  is $C_4$ of length $4$, by the same argument as in the proof of Lemma $\ref{T4}$ we have the following situations:
\begin{enumerate}
\item If $G=C_4$, then $\dstab(I(G)) = 1$.
\smallskip
\item If $G\ne C_4$ and $C_4$ has at least two adjacent vertices of degree $2$ in $G$, then  $\dstab(I(G)) = \upsilon(G)-\varepsilon_0(G)-2$.
\smallskip
\item In the remain cases, $\dstab(I(G))= \upsilon(G)-\varepsilon_0(G)-1$.
\end{enumerate}
\end{rem}

Thus if every connected component of $G$ is either a tree or a unicyclic graph, then we can compute $\dstab(I(G))$ by using Theorem $\ref{MP}$, Lemmas $\ref{T6}$, $\ref{T4}$, $\ref{T5}$ and Remark $\ref{RMP}$.

\subsection*{Acknowledgment} I would like to thank Professors L. T. Hoa and  N. V. Trung for helpful comments. I would like to thank the referee for his/her careful reading and many useful suggestions. A part of this work was carried out while I visited Genoa University under the support from EMMA in the framework of the EU Erasmus Mundus Action $2$. I would like to thank them and Professor A. Conca for support and hospitality. This work is also partially supported by NAFOSTED (Vietnam), Project $101.01-2011.48$.

\end{document}